\documentclass[11pt]{article}
\usepackage{amsmath, amssymb, amsthm, geometry, hyperref}
\usepackage{mathrsfs}
\usepackage{tikz-cd}
\usepackage{graphicx} 

\usepackage[colorinlistoftodos,prependcaption,textsize=tiny]{todonotes}

\newtheorem{thm}{Theorem}[section]
\newtheorem{prop}[thm]{Proposition}
\newtheorem{lem}[thm]{Lemma}
\newtheorem{cor}[thm]{Corollary}
\newtheorem{mainthm}{Theorem}

\theoremstyle{definition}

\newtheorem{rem}[thm]{Remark}
\newtheorem{ex}[thm]{Example}

\newcommand{\abs}[1]{\left\vert#1\right\vert}
\newcommand{\OO}{\mathcal{O}}
\newcommand{\C}{\mathbb{C}}
\newcommand{\PP}{\mathbb{P}}
\newcommand{\Z}{\mathbb{Z}}

\DeclareMathOperator{\sExt}{\mathscr{E}\!xt}

\DeclareMathOperator{\Ext}{Ext}

\DeclareMathOperator{\Pic}{Pic}

\DeclareMathOperator{\Bl}{Bl}

\newcommand{\ev}{\operatorname{ev}}      

\newcommand{\vT}{\mathbb{T}}           

\newcommand{\push}{\pi_*}             

\title{Log Conifold Transitions}
\author{Rodolfo Aguilar}
\date{}
\begin{document}
\maketitle

\begin{abstract}
We define log conifold transitions for Fano threefold pairs of index two and study their deformation theory. Relying on the recent solution to the relative Clemens conjectures in this setting, we construct rational curves with normal bundle $\OO(-1)\oplus \OO(-1)$ by blowing up anchored points on the boundary divisor. Contracting these curves yields a singular space with ordinary double points. We prove that local smoothings of the nodes can be lifted to global first-order deformations, and that the global deformation theory of both the log resolution space and the singular log pair is unconditionally unobstructed. Crucially, the geometry of the boundary del Pezzo surface guarantees this unobstructedness. Furthermore, unlike the classical Calabi-Yau case, the underlying Fano geometry forces the vanishing of global topological balancing conditions, allowing local first-order smoothings of the nodes to be lifted independently.
 As applications, we construct new non-Kähler threefolds via smoothings, we analyze the effective geometry of the smoothed threefolds by determining their Picard groups and proving the persistence of free curves. Finally, we study the Hodge theory of these non-Kähler threefolds.
\end{abstract}

\tableofcontents
\section{Introduction}
The deep connection between Hodge theory and rational curves on quintic threefolds was explored extensively by Clemens \cite{C83, C87}, building upon Griffiths's observations regarding lines and Abel-Jacobi maps \cite{G69}. A classical application of this theory is the conifold transition: a geometric operation in which a collection of disjoint rational curves with normal bundle $\OO(-1)\oplus\OO(-1)$ is contracted to ordinary double points, and the resulting singular threefold is subsequently smoothed \cite{C83b, F86}. As highlighted in the recent survey by Collins \cite{C25}, this construction has an impact well beyond algebraic geometry. For instance, it provides a fundamental mechanism to connect topologically distinct Calabi-Yau threefolds—motivating Reid's fantasy of a universal web—and it plays a central role in differential geometry and string theory by describing topology-changing processes. In this work, we generalize this construction to the logarithmic setting, defining \textit{log conifold transitions} for certain Fano threefold pairs.

We studied curves on log Calabi-Yau threefold pairs in \cite{AGG24b}. Moreover, we showed in \cite{A25} that Hodge theory behaves well for threefold pairs $(X,Y)$ where $K_X+mY\cong \OO_X$. We paid special attention to the half-anticanonical case ($m=2$). In this setting, we established that $H^0(N_{C/X}(-Y))^*\cong H^1(N_{C/X}(-Y))$, which allowed us to formulate the relative Clemens conjectures following \cite{C87}. These conjectures—predicting the finiteness and normal bundles of anchored rational curves—were recently proven by Zahariuc \cite{A26b} for smooth Fano threefolds of index two, utilizing specialization methods \cite{Z18} and moduli irreducibility \cite{LT19}.

With the relative Clemens conjectures resolved, we have a robust supply of rational curves of degree $e$ intersecting $Y$ transversely. By blowing up these intersection points, we obtain rational curves with normal bundle $\OO(-1)\oplus \OO(-1)$. It is important to note that the exceptional $(-1,-1)$ curves are disjoint from the proper transform.
While for low-degrees rational curves, it can be proved easily that the strict transforms are disjoint, along \emph{all this work}, we assume that \emph{the strict transform are disjoint}.
Hence, these can be contracted to yield a singular space $X_0$ with ordinary double points. This gives a new way to construct complex threefolds and degenerations to nodal ones which can not be obtained by elementary methods. We refer to the process of blowing up the anchored intersections, contracting their strict transforms, and smoothing the resulting nodal log pair as a \textit{log conifold transition}.

\subsection{Lifting and Unobstructedness of Log Conifold Transitions}
The primary technical goal of this paper is to study the deformation theory of these singular log pairs. As highlighted by Friedman \cite{F26}, proving smoothability for singular Calabi-Yau and Fano varieties relies heavily on analyzing the compactly supported cohomology of the exceptional locus of the resolution.

Motivated by this philosophy, we study the lifting of local first-order deformations to global ones, and the unobstructedness of global deformations to higher orders. We prove that our log conifold transitions unconditionally satisfy both:
\begin{mainthm} \label{thm:A}
 Let $X$ be a smooth Fano threefold of index two, and let $\sigma:\tilde X\to X$ be the blow-up of $X$ at $e$ generic points $\xi$ on $Y\in \abs{-\frac{1}{2} K_X}$. 
 Let $\pi:\tilde X \to X_0$ be the contraction of a finite number of rational curves with normal bundle $\OO(-1)\oplus \OO(-1)$ arising as the strict transforms of curves passing through $\xi$ transversely. \emph{ We assume these curves are disjoint.}
Then local first-order smoothings of the nodes lift to global first-order deformations, and the global deformations of both the log resolution space and the singular log pair are unconditionally unobstructed: 
\begin{enumerate}
 \item $H^2(\tilde{X}, T_{\tilde X}(-\log \tilde Y))=0$,
 \item $\mathbb{T}_{X_0}^2(-\log Y_0)=0$.
\end{enumerate}
\end{mainthm}

Here we find a crucial difference with the classical Calabi-Yau setting. While classical smoothings are obstructed by global homological cycle relations among the contracted curves, our underlying Fano geometry forces the global target space $H^2(\tilde{X}, T_{\tilde X}(-\log \tilde Y))$ to vanish. Consequently, by Friedman's criterion, the local first-order smoothings of the nodes can be lifted globally and independently, free of any topological balancing conditions. Furthermore, because $H^2(\tilde{X}, T_{\tilde X}(-\log \tilde Y))=0$, these global first-order deformations are unconditionally unobstructed and extend to actual smoothings over the disk. The proof is given in Section \ref{s:Un}.

\subsection{Non-K\"ahler Threefolds and the Effective Cone}
As an application of these unobstructed and independent smoothings, we construct novel examples of non-K\"ahler threefolds. The robust supply of anchored rational curves on Fano threefolds grants us significant freedom to choose specific subsets of curves to contract, generating a vast and flexible family of singular spaces. 

By analyzing the global smoothings of these spaces, we uncover a striking geometric property of the general smoothed fiber $X_t$: a total collapse of the effective cone of divisors coexisting with a massive, persistent abundance of free rational curves. The details are in Section \ref{s:examples}. We summarize these geometric consequences as follows:

\begin{mainthm} \label{thm:B}
With the notation as in Theorem \ref{thm:A}:
\begin{enumerate}
\item Let $(X_0, Y_0)$ be a singular pair obtained by contracting $N_e$ anchored curves of degree $e$ as above. If we completely smooth a proper, non-empty subset of the nodes while performing a small resolution on the rest, the surviving exceptional curves are forced into a trivial homology class, yielding a smooth non-K\"ahler threefold.
 \item If we construct $X_0$ by simultaneously contracting suitably chosen curves of degree $e$ and $e-1$, the Picard rank of $X_0$ drops to $1$. If the number of blown-up points satisfies $e > d+1$, the unique generator of the Picard group becomes globally obstructed. The resulting fully smoothed threefold $X_t$ possesses no effective surfaces.
\item Let $X_0$ be as in the point 2 above. Then general very free rational curves on $X$ entirely avoid the conifold locus and deform unobstructedly to the non-K\"ahler fiber $X_t$.
\end{enumerate}
\end{mainthm}
These spaces sit in contrast to previously known non-K\"ahler constructions. Unlike classical Clemens-Friedman transitions (which eliminate the second Betti number to yield spaces with $b_2=0$) or Poon's spaces (which rely on small resolutions and thus have $b_3=0$), our systematic smoothings naturally jump the third Betti number ($b_3 > 0$) while preserving a positive second Betti number ($b_2 \ge 1$). 

This yields a novel class of non-K\"ahler threefolds characterized by an absence of surfaces coexisting with an abundance of free rational curves. One may ask if this is related to a twistor space. In Section \ref{s:hodge}, we use recent results of \cite{Che24} to conclude that $H^3(X_t)$ has a polarized Hodge structure, see also \cite{F19,Li24}. Using this and the classical results of Clemens-Griffiths for the intermediate Jacobian of the cubic threefold \cite{CG72}, we can conclude that $X_t$ is not even birational to a twistor-space when we start with $X$ a cubic threefold.

The theoretical framework developed here opens several new avenues for exploring the geometry of Fano pairs. The remarkable abundance of curves with normal bundle $\OO(-1)\oplus \OO(-1)$ constructed via these blow-ups naturally suggests deep connections to the birational geometry and minimal model program of logarithmic threefolds, as well as extensions of classical metric smoothings to the boundary setting. We intend to explore these applications in future work.

The paper is organized as follows. In Section \ref{s:prel}, we gather the necessary preliminaries on half-anticanonical pairs and logarithmic deformation theory. The proofs of the main unobstructedness results (Theorem A) are carried out in Section \ref{s:Un}. In Section \ref{s:examples} is dedicated to explicit geometric constructions, including the classical warm-up of projecting a cubic threefold from a point, and the detailed proofs of the non-K\"ahler properties and curve persistence summarized in Theorem B. Finally, in Section \ref{s:hodge}, we study Hodge theoretic properties of these smoothings and give applications.

\section*{Acknowledgements} 
This note arose after attending the conference Spencer Fest at IMSA-Miami. The author is grateful to the organizers, speakers, and participants, and is especially thankful to Phillip Griffiths and Herb Clemens for useful conversations. The author is deeply indebted to Robert Friedman. His talk surveying recent unobstructedness results for singular Fano and Calabi-Yau varieties directly motivated this work. Furthermore, he provided a careful reading of a first version of this note; his comments helped to clarify the deformation arguments as well as to extend the scope of examples and find more structure on them.
The author is also grateful to Richard Thomas for helpful correspondence that improved the clarity of the exposition.

This work was partially supported by a postdoctoral fellowship, Estancias Posdoctorales por México 2025, from the Secretaría de Ciencia, Humanidades, Tecnología e Innovación (SECIHTI), Mexico and by the project SECIHTI \#CBF-2025-I-673.

\section{Preliminaries}\label{s:prel}
\subsection{1/2 log CY threefolds}\label{ss:12logCY}
Let $X$ be a smooth compact threefold and $Y\in \abs{-\frac{1}{2}K_X}$, this is, such that 
$$ K_X+2Y\cong \mathscr{O}_X .$$ 
We call the pair $(X,Y)$ a $\frac{1}{2}$ log CY threefold.

Recall the following Theorem of \cite{A26b}.

\begin{thm}[Aguilar-Zahariuc]\label{thm:AZ} 
Let $X$ be a smooth prime Fano threefold of index two and degree $d\in \{2,3,4,5,8\}$. Let $Y\in \abs{-\frac{1}{2}K_X}$, $C$ a smooth generic rational curve of degree $e$ and denote by $\xi=C\cap Y$. Then 
\begin{enumerate}
  \item There exists only a finite number $C_1, \ldots, C_{N_e}$ of smooth rational curves of degree $e$ passing through $\xi$ with $N_e>0$.
  \item The normal bundle of each $C_i$ is $N_{C_i/X}\cong \mathscr{O}(e-1)\oplus \mathscr{O}(e-1).$
\end{enumerate}
\end{thm}

\begin{lem}\label{lem:BlowHalf}
  Let $X$ be as in the Theorem above. Fix $d\in \{2,3,4,5,8\}$ and $e\in \mathbb{N}$. 
Let $C_1, \ldots, C_{N_e}$ be the $N_e$ rational curves of degree $e$ intersecting $Y$ exactly at $\xi$. 
Let $\sigma: \tilde{X} = \Bl_\xi X \to X$ be the blow-up of $X$ at $\xi$. 
Denote the strict transforms of the rational curves by $\tilde{C}_1, \ldots, \tilde{C}_{N_e}$, that of $Y$ by $\tilde Y$ and the exceptional divisors by $E_1, \ldots, E_{N_e}$. Then 
\begin{enumerate}
  \item The strict transforms have all normal bundle $\OO(-1)\oplus \OO(-1)$, are disjoint from the strict transform $\tilde{Y}$ of $Y$, and \emph{assuming they are disjoint,} by Grauert criterion, they can be contracted.  
  \item The pair $(\tilde X, \tilde Y)$ satisfy $K_{\tilde X}+2\tilde Y \cong \OO_{\tilde X}. $
\end{enumerate}
\end{lem}
\begin{proof}
  The first point follows from Theorem \ref{thm:AZ} and the fact that the curves are transverse to $Y$.
  For the second, we compute:
  $$K_{\tilde X}=\sigma^* K_X+2\sum E_i, \quad \tilde Y=\sigma^* Y-\sum E_i. $$
  Thus, $K_{\tilde X}+2\tilde Y =\sigma^*(K_X+2Y)\cong \OO_{\tilde X}.$
\end{proof}

Denote by $\pi:\tilde{X}\to X_0$ the morphism obtained by contracting a subset of the strict transforms $\tilde{C}_i$ and by $Y_0:=\pi(\tilde{Y})\cong \tilde{Y}$.

The goal of these notes is to study when we can smooth the pair $(X_0,Y_0)$.

\subsection{Deformation of threefolds and pairs}\label{ss:DeformationPairs}
Here we collect some results adapting the deformation theory of \cite{F86} to the logarithmic setting. 

Denote by:
\begin{enumerate}
  \item $X_0$ a compact analytic threefold.
  \item $Z=\{p_1,\ldots, p_j\}$ the singular locus of $X_0$, which consists of isolated Gorenstein singularities admitting a small resolution.
  \item $\pi: \tilde{X}\to X_0$ a small resolution of $X_0$ at all the points in $Z$.
  \item $\Gamma=\pi^{-1}(Z)$ the exceptional locus, which consists of a union of rational curves.
  \item $Y_0\subset X_0$ a smooth divisor such that $Y_0 \cap Z = \emptyset$.
\end{enumerate}

Because the divisor $Y_0$ avoids the singular locus $Z$ (i.e., $Y_0 \cap Z = \emptyset$), we can define the global sheaf of logarithmic K\"ahler differentials $\Omega_{X_0}^1(\log Y_0)$ without requiring generalized cotangent complexes for singular spaces. We do this by defining the sheaf on an open cover, or equivalently, directly on its stalks.

Consider the Zariski open cover of $X_0$ given by $U = X_0 \setminus Z$ (the smooth locus) and $V = X_0 \setminus Y_0$. 
On $U$, the space is smooth and $Y_0 \subset U$ is a smooth divisor, so the classical Deligne sheaf of logarithmic differentials $\Omega_U^1(\log Y_0)$ is well-defined. 
On $V$, there is no boundary divisor, so we utilize the standard sheaf of K\"ahler differentials $\Omega_V^1$. 
On the intersection $U \cap V = X_0 \setminus (Z \cup Y_0)$, the space is smooth and the boundary is empty. Consequently, both local sheaves canonically restrict to the regular cotangent sheaf $\Omega_{U \cap V}^1$. Because they naturally agree on the overlap, they glue to define a unique global coherent sheaf $\Omega_{X_0}^1(\log Y_0)$ on $X_0$.

Equivalently, this sheaf is completely characterized by its stalks at any point $x \in X_0$:
$$
\left( \Omega_{X_0}^1(\log Y_0) \right)_x = 
\begin{cases} 
\Omega_{X_0, x}^1(\log Y_0) & \text{if } x \in Y_0 \text{ (where } X_0 \text{ is guaranteed to be smooth)}, \\
\Omega_{X_0, x}^1 & \text{if } x \notin Y_0 \text{ (which includes all singular points } p_i \in Z).
\end{cases}
$$
By construction, because $\Omega_{X_0}^1(\log Y_0)$ is isomorphic to the standard K\"ahler differentials in an open neighborhood of $Z$, the local deformation sheaves $T_{X_0}^i(-\log Y_0) := \mathcal{E}xt^i(\Omega_{X_0}^1(\log Y_0), \OO_{X_0})$ are unaffected by the boundary $Y_0$ near the singular locus.
Following Friedman's notation, let $T_{X_0}^0(-\log Y_0) = \mathcal{H}om(\Omega_{X_0}^1(\log Y_0), \OO_{X_0})$ be the logarithmic tangent sheaf, and define the local deformation sheaves of the pair as:
\[
T_{X_0}^i(-\log Y_0) := \mathcal{E}xt^i(\Omega_{X_0}^1(\log Y_0), \OO_{X_0}).
\]
Because $\Omega_{X_0}^1(\log Y_0)$ is locally free away from $Z$ (since $Y_0$ is smooth) and isomorphic to $\Omega_{X_0}^1$ away from $Y$, the higher local Ext sheaves are supported strictly on $Z$. Consequently, for $i > 0$, the boundary condition vanishes locally, yielding a canonical isomorphism:
\[
T_{X_0}^i(-\log Y_0) \cong T_{X_0}^i \cong \bigoplus_{k=1}^j T_{p_k}^i
\]
where $T_{p_k}^i$ is the local deformation space of the singularity $p_k$.

The global deformation spaces of the pair, denoted $\mathbb{T}_{X_0}^i(-\log Y_0) = \operatorname{Ext}^i(\Omega_{X_0}^1(\log Y_0), \OO_{X_0})$, are related to the local deformations via the local-to-global spectral sequence:
\[
E_2^{p,q} = H^p(X_0, T_{X_0}^q(-\log Y_0)) \Rightarrow \mathbb{T}_{X_0}^{p+q}(-\log Y_0).
\]
Similarly, the deformations of the log resolution $(\tilde{X}, \tilde{Y})$ are governed by $H^i(\tilde{X}, T_{\tilde{X}}(-\log \tilde{Y}))$. The Leray spectral sequence for the resolution $\pi$ yields:
\[
E_2^{p,q} = H^p(X_0, R^q \pi_* T_{\tilde{X}}(-\log \tilde{Y})) \Rightarrow H^{p+q}(\tilde{X}, T_{\tilde{X}}(-\log \tilde{Y})).
\]
Because $\pi$ is an isomorphism outside $Z$ and $\tilde{Y}$ avoids $\Gamma$, there is a natural identification $\pi_* T_{\tilde{X}}(-\log \tilde{Y}) \cong T_{X_0}(-\log Y)$, allowing us to compare the $E_2$ pages of these spectral sequences directly.

\subsection{Local Cohomology of the Log Tangent Sheaf}
Let $(X_0,0)$ be the germ of an isolated Gorenstein threefold singularity, and $\pi:\tilde{X}\to X_0$ a small resolution. Denote the exceptional locus by 
$$\pi^{-1}(0)=\cup_i C_i=\Gamma,$$
where the $C_i$ are the irreducible curves. 
Recall that by definition of $\Omega_{X_0}^1(\log Y_0)$, it coincides with $\Omega_{X_0}^1$ around a point $0\in X_0$ away from $Y_0$.
\begin{lem}[{\cite[Lemma 3.1]{F86}}]\label{lem:log_tangent_pushforward} $R^0\push T_{\tilde{X}}(-\log \tilde{Y})\cong T^0_{X_0}(-\log Y_0)\cong T_{X_0}^0 .$
\end{lem}

\begin{lem}[{\cite[Lemma 3.2]{F86}}] \label{lem:log_higher_direct_images}
The local cohomology group $H^2_{\{0\}}(R^1\push T_{\tilde{X}}(-\log \tilde{Y}))=0$
\end{lem}

\begin{lem}[{\cite[Lemma 3.3]{F86}}]\label{lem:Schlessinger} $H^0(T_{X_0}^1(-\log Y_0))=H^2_{\{0\}}(T_{X_0}^0(-\log Y_0))=H^2_{\{0\}}(T_{X_0}^0) $  
\end{lem}

\subsection{The Local-to-Global Spectral Sequences}
Let us now come back to the global setting and use the notation of subsection \ref{ss:DeformationPairs}.

\begin{prop}
  The Leray spectral sequence for the log resolution $H^p(R^q \push T_{\tilde{X}}(-\log \tilde{Y}))\Rightarrow H^{p+q}(T_{\tilde{X}}(-\log \tilde{Y}))$ and the local-to-global spectral sequence for the singular pair $\mathbb{T}_{X_0}^i(-\log Y_0)$ induce the following commutative diagram with exact rows:
\begin{equation} \label{eq:friedman_diagram}
\resizebox{\displaywidth}{!}{%
\begin{tikzcd}[column sep=1.5em, row sep=3em, ampersand replacement=\&]
0 \arrow[r] \& H^1( T^0_{X_0}(-\log Y_0)) \arrow[r] \arrow[d, "\cong"] \& H^1(T_{\tilde{X}}(-\log \tilde{Y})) \arrow[r] \arrow[d, "\phi"] \& H^0(R^1\push T_{\tilde{X}}(-\log \tilde{Y})) \arrow[r, "ob"] \arrow[d] \& H^2(T^0_{\tilde{X}}(-\log \tilde{Y})) \arrow[r] \arrow[d, "\cong"] \& H^2(T_{\tilde{X}}(-\log \tilde{Y})) \arrow[d] \arrow[r] \& 0 \\
0 \arrow[r] \& H^1( T^0_{X_0}(-\log Y_0)) \arrow[r] \& \vT^1_{X_0}(-\log Y_0) \arrow[r] \& H^0(X_0, T^1_{X_0}(-\log Y_0)) \arrow[r, "ob"] \& H^2(T^0_{X_0}(-\log Y_0)) \arrow[r] \& \vT^2_{X_0}(-\log Y_0) \arrow[r] \& 0
\end{tikzcd}%
}
\end{equation}
\end{prop}
\begin{proof}
The top row is the standard five-term exact sequence of low degrees associated to the Leray spectral sequence $E_2^{p,q} = H^p(X_0, R^q\push T_{\tilde{X}}(-\log \tilde{Y})) \Rightarrow H^{p+q}(\tilde{X}, T_{\tilde{X}}(-\log \tilde{Y}))$. Because $R^2\push = 0$, the $E_2^{0,2}$ term vanishes, allowing the sequence to terminate exactly at $H^2(\tilde{X}, T_{\tilde{X}}(-\log \tilde{Y}))$.

The bottom row is the analogous five-term exact sequence for the local-to-global spectral sequence $E_2^{p,q} = H^p(X_0, T^q_{X_0}(-\log Y_0)) \Rightarrow \vT^{p+q}_{X_0}(-\log Y_0)$. Because ordinary double points are local complete intersections, $T^2_{X_0} = 0$, allowing the sequence to terminate at $\vT^2_{X_0}(-\log Y_0)$. 

Furthermore, because the singularities are isolated points and the resolution is small, both $R^1\push T_{\tilde{X}}(-\log \tilde{Y})$ and $T^1_{X_0}(-\log Y_0)$ are skyscraper sheaves supported entirely on the zero-dimensional locus $Z$. Consequently, their first cohomology groups vanish 
($H^1(X_0, R^1\push T_{\tilde{X}}(-\log \tilde{Y}))=0$ and $H^1(X_0, T^1_{X_0}(-\log Y_0))=0$), which forces the maps originating from the degree-two base spaces to be surjective.
\end{proof}

\section{Unobstructedness}\label{s:Un}
\subsection{Setting and basic sequences}\label{ss:setting}
Let us fix the setting and notation for this section. Let $X$ be a smooth prime Fano threefold of index $2$ and degree $d\in \{2,3,4,5,8\}$.
Let $Y\in \abs{-\frac{1}{2} K_X}$ be a generic smooth divisor and $\xi\subset Y$ a generic subset of $e$ points.
Denote by $\sigma:\tilde{X}\to X$ the blow-up of $X$ at $\xi$ and by $\tilde{Y}$ the strict transform of $Y$, by $\Gamma=\cup \tilde{C}_i\subset \tilde X\setminus \tilde Y$ a union of smooth rational curves assumed \emph{disjoint} with normal bundle $N_{\tilde C_i/\tilde X}\cong \OO(-1)\oplus \OO(-1)$.
Let $\pi:\tilde{X}\to X_0$ be the contraction of $\Gamma$ with $\pi_*(\Gamma)=Z$ and $Y_0\cong \pi_*(\tilde Y)$.

For deformations of the threefold $X_0$ we follow the notation of \cite{F86}. 
As for deformations of the pair $(X_0,Y_0)$ we use the notation and definitions of section \ref{ss:DeformationPairs}. 

There is the exact sequence
$$0\to T_{X_0}^0(-\log Y_0)\to T_{X_0}^0\to N_{Y_0/X_0}\to 0. $$
As $N_{Y_0/X_0}=K_{Y_0}^{-1}$, $H^2(Y_0, N_{Y_0/X_0})=H^2(Y_0, K_{Y_0}^{-1})$ is Serre dual to $H^0(Y_0, K_{Y_0}^2)=0$. 
So there is an exact sequence
\begin{equation}\label{eq:LEStanToNor}
  H^1(X_0, T_{X_0}^0(-\log Y_0))\to H^1(X_0, T_{X_0}^0)\to H^1(Y_0, N_{Y_0/X_0})\to H^2(X_0, T_{X_0}^0(-\log Y_0))\to 0. 
\end{equation}
The local-to-global $\Ext$ spectral sequence gives:
{\small 
\begin{equation} \label{eq:locToGlo} 
\begin{gathered}
\begin{tikzcd}[column sep=0.5em] 
\Ext^1(\Omega_{X_0}^1(\log Y_0), \OO_{X_0}) \ar[d, "\cong "] \ar[r] & H^0(\sExt^1(\Omega_{X_0}^1(\log Y_0), \OO_{X_0}))\ar[d, "\cong "] \ar[r]  & H^2(X_0, T_{X_0}^0(-\log Y_0)) \ar[r] & \Ext^2(\Omega_{X_0}^1(\log Y_0), \OO_{X_0}) \ar[r] \ar[d, "\cong"] & 0 \\
\mathbb{T}_{X_0}^1(-\log Y_0)  &  H^0(T_{X_0}^1) &      & \mathbb{T}_{X_0}^2(-\log Y_0) & 
\end{tikzcd}
\end{gathered}
\end{equation}
}
Similarly, the Poincaré residue exact sequence 
$$0\to \Omega_{X_0}^1\to \Omega_{X_0}^1(\log Y_0)\to \OO_{Y_0}\to 0, $$
gives the long exact $\Ext$ sequence
\begin{equation}\label{eq:DefPairToDef}
  \begin{tikzcd}
  \Ext^1(\Omega_{X_0}^1(\log Y_0),\OO_{X_0}) \ar[d, "\cong "] \ar[r] & \Ext^1(\Omega_{X_0}^1,\OO_{X_0}) \ar[d, "\cong"] \ar[r] & \Ext^2(\OO_{Y_0}, \OO_{X_0})\ar[d, "\cong"] \ar[r] & 0\\
  \mathbb{T}_{X_0}^1(-\log Y_0) & \mathbb{T}_{X_0}^1 & H^1(Y_0,N_{Y_0/X_0}) &
\end{tikzcd}.
\end{equation}
Here $\Ext^2(\OO_{Y_0}, \OO_{X_0})\cong H^1(\sExt^1(\OO_{Y_0}, \OO_{X_0}))\cong H^1(\OO_{X_0}(Y_0)/\OO_{X_0})\cong H^1(Y_0, N_{Y_0/X_0}).$
\subsection{Deformations of the threefold}

\begin{prop}\label{prop:defAbs} The deformations of $X_0$ and of $\tilde X$ are unobstructed, this is, 
  \begin{enumerate}
  \item $\mathbb{T}_{X_0}^2=0$,
  \item $H^2(X_0,T_{X_0}^0)\cong H^2(\tilde X, T_{\tilde X})=0$.
  \end{enumerate} 
  As a consequence, the map $\mathbb{T}_{X_0}^1\to H^0(X_0, T_{X_0}^1)$ is surjective.
\end{prop}
\begin{proof}
  We have the following exact sequence given by the local-to-global spectral sequence:
  $$\mathbb{T}_{X_0}^1\to H^0(X_0,T_{X_0}^1)\to H^2(X_0, T_{X_0}^0)\to \mathbb{T}_{X_0}^2\to 0. $$
  Because $Z$ consists of ordinary double points, we have that $\push T_{\tilde X}\to T_{X_0}^0$ is an isomorphism and $R^1\push T_{\tilde X}=R^2\push T_{\tilde X}=0.$
  Thus, from the Leray spectral sequence, we have that $H^2(X_0,T_{X_0}^0)\cong H^2(\tilde X, T_{\tilde X})$. 

  Also, we have that $R^1\sigma_* T_{\tilde X}=R^2\sigma_* T_{\tilde X}=0$ and, from the natural map $\sigma_* T_{\tilde X}\to T_X$, there is an exact sequence:
  $$0\to \sigma_* T_{\tilde X}\to T_X\to \mathcal{S}\to 0, $$
  where $S$ is a skyscraper sheaf supported at the blowup points $\xi$. In particular, via Leray and the vanishing of $H^i(X,\mathcal{S})$ for $i>0$, 
  $$H^2(X_0, T_{X_0}^0)\cong H^2(\tilde X, T_{\tilde X})\cong H^2(X, \sigma_* T_{\tilde X})\cong H^2(X, T_X)=0, $$
  since $X$ is Fano and via Nakano vanishing. So finally $\mathbb{T}_{X_0}^2=0$ and the map $\mathbb{T}_{X_0}^1\to H^0(X_0, T_{X_0}^1)$ is surjective. This surjectivity guarantees that any local first-order smoothing lifts to a global first-order deformation. Combined with the vanishing of the obstruction space $\mathbb{T}_{X_0}^2(-\log Y_0) = 0$, these deformations are unconditionally unobstructed and extend to higher orders.
\end{proof}

\subsection{Deformations of the pair: del Pezzo boundary}
\begin{prop}
  Assume that $Y_0$ is a del Pezzo surface. Hence
  $$\Ext^2(\Omega_{X_0}^1(\log Y_0),\OO_{X_0})\cong \mathbb{T}_{X_0}^2(-\log Y_0)=0. $$
\end{prop}
\begin{proof}
  It follows from 
  $$H^1(Y_0, N_{Y_0/X_0})=H^2(Y_0, N_{Y_0/X_0})=0 $$
  by using Kodaira vanishing.

  We conclude the proof by using the exact sequences (\ref{eq:LEStanToNor}) in subsection \ref{ss:setting}, which imply that $H^2(X_0, T_{X_0}^0(-\log Y_0))=0$, from which the result follows easily.
\end{proof}

Using \ref{eq:locToGlo} and \ref{eq:DefPairToDef} we obtain the following Corollary.

\begin{cor}
  The natural map $\mathbb{T}_{X_0}^1(-\log Y_0)\to H^0(T_{X_0}^1)$ is surjective, and the deformations of the pair $(X_0,Y_0)$ are unobstructed.
  
  Moreover, $\mathbb{T}_{X_0}^1(-\log Y_0)\to \mathbb{T}_{X_0}^1$ is surjective and the morphism from the deformation functor of the pair $(X_0,Y_0)$ to the deformation functor of $X_0$ is smooth.
\end{cor}

\begin{rem}
  In general, we don't expect $Y_0$ to be a stable submanifold of $X_0$. For example, a deformation of $\tilde X$ and hence $X_0$ allows for moving the points of $\xi$ arbitrarily. 
  But once $e$ is large enough, there won't be a $Y\in \abs{-\frac{1}{2} K_X}$ passing through $\xi$. This is reflected in the fact that, by Riemann-Roch, $\xi(Y_0, K_{Y_0}^{-1})=c_1(K_{Y_0})^2+1$, so that $h^1(Y_0, K_{Y_0}^{-1})\not = 0$ as soon as $c_1(K_{Y_0})^2\leq -2$.
\end{rem}

\begin{rem}
If we blow-up less points than the degree of $X$, this is $e<d$, we can show that $H^0(-K_{\tilde {X}})$ is non-empty, basepoint free, its generic element $S$ is smooth, disjoint from $\Gamma$ and by Kawamata-Viehweg $H^1(S,N_{S/\tilde X})=0$. Hence the log-Calabi-Yau pair $(X_0,S)$ is also unobstructed.
\end{rem}

\subsection{Deformations of the pair: general case}
\begin{thm}\label{thm:unobstructedness}
  Let $(X_0,Y_0)$ be as in subsection \ref{ss:setting}. Then
  \begin{enumerate}
    \item $\mathbb{T}_{X_0}^2(-\log Y_0)=0$,
    \item the map $\mathbb{T}_{X_0}^1(-\log Y_0)\to H^0(T_{X_0}^1)$ is surjective,
    \item $H^2(\tilde{X}, T_{\tilde X}(-\log \tilde Y))=0.$
  \end{enumerate}
   
\end{thm}
\begin{proof}
By using the Leray spectral sequence as in the proof of Proposition \ref{prop:defAbs} and by diagram (\ref{eq:friedman_diagram}), it suffices to show that $H^2(\tilde{X}, T_{\tilde X}(-\log \tilde Y))=0.$

By considering part of the long exact sequence associated to
$$0\to T_{\tilde X}(-\log \tilde Y)\to T_{\tilde X}\to N_{\tilde Y/ \tilde X}\to 0,$$
and using Proposition \ref{prop:defAbs}, the vanishing $H^2(\tilde{X}, T_{\tilde X}(-\log \tilde Y))=0$ is equivalent to show that 
$$H^1(\tilde X, T_{\tilde X})\to H^1(\tilde Y, N_{\tilde Y/ \tilde X}) $$
is surjective, because $$H^2(\tilde X, T_{\tilde X})\cong H^2(X, \sigma_* T_{\tilde X})\cong H^2(T_X)=0$$
since $X$ is Fano and via Nakano vanishing. 

Denote the ideal sheaf on $X$ (respectively on Y) of the points $\xi$ by $I_{\xi/X}$ (resp. by $I_{\xi/Y}$).
We have isomorphisms
$$H^1(\tilde X, T_{\tilde X})\cong H^1(X, \sigma_* T_{\tilde X})\cong H^1(X, T_X\otimes I_{\xi/X}). $$
using this and the coboundary map $\partial$ given by the following sequence:
$$0\to T_X\otimes I_{\xi/X}\to T_X\to N_{\xi/X}\to 0, $$
we obtain $\partial: H^0(\xi, N_{\xi/X})\to H^1(\tilde X, T_{\tilde X})$. Its image consists of the tangents to deformations on $\tilde X$ arising from keeping $X$ fixed but deforming the blowup points $\xi$.

Let $\bar{\sigma}:=\sigma\vert_{\tilde Y}$ and let $\bar{E}_i =E_i\cap \tilde{Y}$ be the restrictions of the exceptional divisors $E_i$.
Since $N_{\tilde Y/\tilde X}\cong \bar{\sigma}^* N_{Y/X}\otimes \sum_i \OO_{\tilde Y}(-\bar E_i)$, we have $\bar{\sigma}_* N_{\tilde{Y}/\tilde{X}}=N_{Y/X}\otimes I_{\xi/Y}$, and $R^i \bar{\sigma}_* N_{\tilde Y/\tilde X}=0$ for $i>0$.
Twisting the ideal sheaf sequence defining $\xi$ inside $Y$, we obtain a coboundary map:
$$\bar{\partial}:H^0(\xi, N_{Y/X}\otimes \OO_{\xi})\to H^1(Y, N_{Y/X}\otimes I_{\xi/Y})=H^1(Y, \bar{\sigma}_* N_{\tilde Y/\tilde X})=H^1(\tilde Y, N_{\tilde Y/ \tilde X}). $$
In the present case, $H^1(Y, N_{Y/X})=0$ since $Y$ is a del Pezzo surface and $N_{Y/X}=K_{Y}^{-1}$.
Now, we have a commutative diagram 
$$\begin{tikzcd}
  H^0(\xi,N_{\xi/X}) \ar[r] \ar[d, "\partial"] & H^0(\xi, N_{Y/X}\otimes \OO_\xi) \ar[d, "\bar{\partial} "] \\
  H^1(\tilde X, T_{\tilde X}) \ar[r] & H^1(\tilde Y, N_{\tilde Y/\tilde X})
\end{tikzcd}. $$
The top horizontal map comes from the sequence of normal bundles 
$$0\to N_{\xi/Y}\to N_{\xi/X}\vert_{\xi} \to N_{Y/X}\otimes \OO_{\xi}\to 0.$$ 
This sequence consists of skyscraper sheaves supported on $\xi$. Then $H^0(\xi, N_{\xi/X})\to H^0(\xi, N_{Y/X}\otimes \OO_{\xi})$ is surjective and $\bar{\partial}$ is surjective, so $H^1(\tilde X, T_{\tilde X})\to H^1(\tilde Y, N_{\tilde Y/\tilde X})$ is surjective.
\end{proof}

\begin{cor}
  $\tilde Y$ is a stable submanifold of $\tilde X$ if $e\leq d+1$.
\end{cor}
\begin{proof}
  The argument of the previous Theorem shows that $H^1(\tilde Y, N_{\tilde Y/\tilde X})=0$ if and only if the restriction map $H^0(Y, N_{Y/X})\to H^0(\xi, N_{Y/X}\otimes \OO_\xi)$ is surjective. 
  We have that $H^0(Y, N_{Y/X})=\mathbb{C}^{d+1}$ and  $H^0(\xi, N_{Y/X}\otimes \OO_\xi)=\mathbb{C}^e$
\end{proof}

To understand the geometric nature of the local deformations we are lifting, we utilize the local cohomology supported on the exceptional locus $\Gamma = \pi^{-1}(Z)$. 

\begin{lem} \label{lem:local_CY_identification}
There is a canonical surjection from the local deformation space of the nodes to the vector space spanned by the fundamental classes of the contracted curves:
\[
H^0(X_0, T^1_{X_0}(-\log Y_0)) \twoheadrightarrow H^2_\Gamma(\tilde{X}, \Omega^2_{\tilde{X}}).
\]
\end{lem}
\begin{proof}
By the generalized Schlessinger theorem (Lemma \ref{lem:Schlessinger}), the local deformation space is canonically isomorphic to the local cohomology of the tangent sheaf: 
$$H^0(X_0, T^1_{X_0}(-\log Y_0)) \cong H^2_Z(X_0, T^0_{X_0}(-\log Y_0)).$$ 

We analyze the Leray spectral sequence for local cohomology with supports: $$E_2^{p,q} = H_Z^p(X_0, R^q \pi_* T_{\tilde{X}}(-\log \tilde{Y})) \implies H^{p+q}_\Gamma(\tilde{X}, T_{\tilde{X}}(-\log \tilde{Y})).$$ Because $R^1\pi_* T_{\tilde{X}}(-\log \tilde{Y})$ is a skyscraper sheaf supported on the zero-dimensional locus $Z$, its higher local cohomology vanishes, specifically $H^1_Z(X_0, R^1\pi_* T_{\tilde{X}}(-\log \tilde{Y})) = 0.$ This yields a standard edge surjection:
\[
H^2_Z(X_0, T^0_{X_0}(-\log Y_0)) \twoheadrightarrow H^2_\Gamma(\tilde{X}, T_{\tilde{X}}(-\log \tilde{Y})).
\]

To identify this target space, we use the $\frac{1}{2}$-log Calabi-Yau condition. Because $K_{\tilde{X}} + 2\tilde{Y} \sim 0$, there is a canonical global isomorphism of sheaves $T_{\tilde{X}}(-\log \tilde{Y}) \cong \Omega_{\tilde{X}}^2(\log \tilde{Y})(\tilde{Y})$. 
Crucially, the exceptional curves $\Gamma$ are strictly disjoint from the boundary divisor $\tilde{Y}$. Therefore, in a formal neighborhood of $\Gamma$, the logarithmic poles and the twist become trivial, yielding a local identification $T_{\tilde{X}}(-\log \tilde{Y})|_\Gamma \cong \Omega^2_{\tilde{X}}|_\Gamma$. 

Taking local cohomology with supports in $\Gamma$, we obtain the isomorphism:
\[
H^2_\Gamma(\tilde{X}, T_{\tilde{X}}(-\log \tilde{Y})) \cong H^2_\Gamma(\tilde{X}, \Omega^2_{\tilde{X}}).
\]
By standard Hodge theory, $H^2_\Gamma(\tilde{X}, \Omega^2_{\tilde{X}})$ is exactly the complex vector space generated by the fundamental classes of the irreducible components of $\Gamma$. Composing the Leray edge surjection with this isomorphism yields the result.
\end{proof}

\begin{cor} \label{cor:independent_smoothings}
There exist global smoothings $(X_t, Y_t)$ of the singular log pair $(X_0, Y_0)$. Furthermore, the singular nodes $p \in Z$ can be smoothed entirely independently of one another without requiring any topological balancing conditions among the contracted curves.
\end{cor}
\begin{proof}
By Lemma \ref{lem:local_CY_identification}, the local deformation space $H^0(X_0, T^1_{X_0}(-\log Y_0))$ surjects onto the fundamental classes of the curves. This guarantees that within the space of formal local deformations of the nodes, there exist directions corresponding precisely to the standard conifold smoothings of the rational curves. 

Because Theorem \ref{thm:unobstructedness} establishes that $H^2(X_0, T^0_{X_0}(-\log Y_0)) = 0$, the exact sequence local-to-global dictates that the map:
\[
\vT^1_{X_0}(-\log Y_0) \twoheadrightarrow H^0(X_0, T^1_{X_0}(-\log Y_0))
\]
is strictly surjective. Therefore, any combination of local conifold smoothings at the individual nodes lifts to an unobstructed global deformation of the pair. Because the global target space $H^2(\tilde{X}, T_{\tilde{X}}(-\log \tilde{Y}))$ inherently vanishes due to the Fano geometry, there are no homological cycle relations connecting the components of $\Gamma$, allowing the smoothings to be chosen independently.
\end{proof}

\begin{prop}\label{prop:smoothingHalfPreservation} For the construction of Section \ref{ss:12logCY}, 
the smooth general fiber $(X_t, Y_t)$ of the smoothing satisfies 
\[ K_{X_t} + 2Y_t \cong \OO_{X_t}. \]
\end{prop}

\begin{proof}
By Lemma \ref{lem:BlowHalf}, the resolution $\tilde{X}$ satisfies $K_{\tilde{X}} + 2\tilde{Y} \cong \OO_{\tilde{X}}$. Since the curves contracted by $\pi: \tilde{X} \to X_0$ are disjoint from the strict transform $\tilde{Y}$, and ODPs are Gorenstein singularities, the canonical divisor $K_{X_0}$ is well-defined as a Cartier divisor. The linear equivalence on $\tilde{X}$ descends to the singular fiber, yielding $K_{X_0} + 2Y_0 \cong \OO_{X_0}$. 

Now, let $(\mathcal{X}, \mathcal{Y}) \to \Delta$ be a flat, $1$-parameter smoothing of the pair $(X_0, Y_0)$. Since $X_0$ has only ordinary double points, the total space $\mathcal{X}$ is smooth in a neighborhood of the singular locus $Z$. We define the relative line bundle:
\[ \mathcal{L} = \OO_{\mathcal{X}}(K_{\mathcal{X}/\Delta} + 2\mathcal{Y}). \]
By construction, the restriction to the central fiber is trivial: $\mathcal{L}|_{X_0} \cong \OO_{X_0}$. 

To prove $\mathcal{L}$ remains trivial on nearby fibers, we examine the cohomology of the structure sheaf. Since $X$ is a Fano threefold, $H^1(X, \OO_X) = 0$, and thus $H^1(\tilde{X}, \OO_{\tilde{X}}) = 0$ by the blow-up formula. Furthermore, as ODPs are rational singularities, the higher direct images of the structure sheaf vanish, i.e., $R^i\pi_* \OO_{\tilde{X}} = 0$ for $i > 0$. The Leray spectral sequence then implies $H^1(X_0, \OO_{X_0}) \cong H^1(\tilde{X}, \OO_{\tilde{X}}) = 0$.

By the semicontinuity theorem, $h^1(X_t, \OO_{X_t}) = 0$ for sufficiently small $t$. This implies that the first Chern class map $c_1: \Pic(\mathcal{X}/\Delta) \to R^2\pi_* \underline{\Z}$ is injective. Since the section $c_1(\mathcal{L})$ is locally constant and vanishes at $t=0$, it must be zero for all $t \in \Delta$. Consequently, $\mathcal{L}$ is trivial in the relative Picard group, which implies $\mathcal{L}|_{X_t} \cong \OO_{X_t}$ for all $t$.
\end{proof}

\section{Examples}\label{s:examples}
\subsection{Projection from a point and partial smoothings}

\begin{ex} \label{ex:cubic_to_double_solid}
Let $X \subset \PP^4$ be a smooth cubic threefold and $Y \in \abs{H}$ a generic hyperplane section. Let $y \in Y$ be a generic point. It is well-known that there are exactly six lines $L_1, \dots, L_6$ in $X$ passing through $y$. Since $y \in Y$ is generic, these lines are transverse to $Y$ and thus represent anchored rational curves of degree $e=1$.

Let $\sigma: \tilde{X} \to X$ be the blow-up of $X$ at $y$ with exceptional divisor $E \cong \PP^2$. The strict transforms $\tilde{L}_1, \dots, \tilde{L}_6$ are disjoint smooth rational curves with normal bundle $N_{\tilde{L}_i/\tilde{X}} \cong \OO(-1) \oplus \OO(-1)$. The linear system $\abs{H-E}$ on $\tilde{X}$ defines a morphism $\phi: \tilde{X} \to \PP^3$, which is the resolution of the linear projection from the point $y \in \PP^4$. 

Geometrically, $\phi$ is a $2:1$ cover of $\PP^3$. For any point $p \in \PP^3$, the line in $\PP^4$ passing through $y$ and $p$ intersects the cubic $X$ at $y$ and two additional points (counting multiplicity); the map $\phi$ identifies these two points. The branch locus of this cover is a quartic surface $B \subset \PP^3$. The six lines passing through $y$ correspond precisely to the six points in $\PP^3$ where the projection is not finite (the fibers are the entire lines). 

By contracting the strict transforms $\tilde{L}_i$, we obtain a singular threefold $X_0$ with six ordinary double points. This $X_0$ is the double cover of $\PP^3$ branched along a quartic surface with six nodes. A global smoothing of $(X_0, Y_0)$ yields a smooth threefold $X_2$ which is a double cover of $\PP^3$ branched along a smooth quartic surface. Thus, $X_2$ is a smooth quartic double solid. In this way we recover the classical projection from a point from the viewpoint of log-conifolds transitions.
The same projection from a point send the del Pezzo threefold of degree $d$ to one degree less. 
\end{ex}
\begin{ex}
Let $X$ be a smooth prime Fano threefold of index $2$ and degree $d \in \{2,3,4,5,8\}$. Let $Y \in \abs{-\frac{1}{2} K_X}$ be a generic smooth divisor, and let $C$ be a generic rational curve of degree $e$ intersecting $Y$ transversely. We denote this intersection by $\xi := C \cap Y$. 

By Theorem \ref{thm:AZ}, there is a finite number $N_e$ of rational curves $C_1, \dots, C_{N_e}$ of degree $e$ intersecting $Y$ exactly at the points $\xi$. Let $\sigma: \tilde{X} = \Bl_\xi X \to X$ be the blow-up of $X$ at $\xi$, and let $\tilde{C}_1, \dots, \tilde{C}_{N_e}$ denote the strict transforms of these curves. As previously established, each strict transform has normal bundle:
\[ N_{\tilde{C}_i/\tilde{X}} \cong \OO(-1) \oplus \OO(-1). \]
Assuming disjointness, they can be simultaneously contracted to yield a singular threefold $X_0$ with $N_e$ ordinary double points. 

Let $S \subset \{1, \dots, N_e\}$ be a choice of subset. By our independent smoothing results Corollary \ref{cor:independent_smoothings}, we can completely smooth the nodes of $X_0$ corresponding to the indices in $S$. In order to obtain a completely smooth threefold, we perform a small resolution on the remaining nodes corresponding to the complement $\{1, \dots, N_e\} \setminus S$ (which geometrically corresponds to leaving the curves $\tilde{C}_j$ for $j \notin S$ intact without contracting them). Denote this resulting smooth threefold by $X_S$.

\begin{prop}
If $N_e > 1$ and $S$ is a non-empty proper subset of $\{1, \dots, N_e\}$, then the smooth threefold $X_S$ is not Kähler. 
\end{prop}

\begin{proof}
In the resolution space $\tilde{X}$, the strict transforms $\tilde{C}_1, \dots, \tilde{C}_{N_e}$ all share the same homology class $[\tilde{C}] \in H_2(\tilde{X}, \Z)$, as they are strict transforms of curves of the same degree passing through the exact same base locus $\xi$. 

During the conifold transition to $X_S$, the curves $\tilde{C}_i$ for $i \in S$ are contracted to nodes and then smoothed. A fundamental topological consequence of this local smoothing is that the homology class of the contracted curves becomes trivial in the smoothed manifold; that is, the relation $[\tilde{C}] = 0$ is enforced in $H_2(X_S, \Z)$.

However, because $S$ is a proper subset, there is at least one index $j \notin S$. The curve $\tilde{C}_j$ is never contracted and survives as a well-defined, effective holomorphic curve inside $X_S$. Its homology class in $X_S$ is precisely the image of the shared class $[\tilde{C}]$, which we just established is zero. 

If $X_S$ were a Kähler manifold, it would possess a Kähler metric with an associated closed Kähler form $\omega > 0$. The volume of the effective curve $\tilde{C}_j$ would be strictly positive, $\int_{\tilde{C}_j} \omega > 0$. But since $[\tilde{C}_j] = 0$ in homology, Stokes' theorem dictates that this integral must evaluate to zero, a contradiction. Therefore, $X_S$ cannot be Kähler.
\end{proof}
\end{ex}

\subsection{Systematic smoothings and free curves}
\begin{prop}
Let $X$ be a smooth prime Fano threefold of index two and degree $d \in \{2,3,4,5,8\}$. Let $\tilde{X}$ be the blow-up of $X$ at $e$ points $\xi \subset Y$ lying transversely on $N_e$ rational curves $C_i$ of degree $e$. Let $\pi \colon \tilde{X} \to X_0$ be the conifold contraction of their strict transforms $\tilde{C}_i$ \emph{assumed disjoint}, and let $X_t$ be a generic global smoothing of $X_0$ where all $N_e$ nodes are smoothed. Then:
\begin{enumerate}
    \item The Hodge number $h^{1,2}(X_t)$ satisfies $h^{1,2}(X_t) = h^{1,2}(X) + N_e - 1$. The Picard group $\Pic(X_t)$ has rank $e$ and is naturally identified with the subgroup:
    \[
    \Pic(X_t) \cong \left\{ aH - \sum_{j=1}^{e} b_j E_j \;\middle|\; ae = \sum_{j=1}^{e} b_j \right\} \subset \Pic(\tilde{X})
    \]
    where $H$ is the pullback of the ample generator of $\Pic(X)$ and $E_1, \dots, E_e$ are the exceptional divisors.
    
\item If $e > d$, the strict transform $\tilde{Y} = H - \sum_{j=1}^{e} E_j$ becomes rigid and does not move in its linear system. Furthermore, the divisor classes $H - eE_j$ are not effective, meaning their linear systems are empty.    
\end{enumerate}
\end{prop}

\begin{proof}
(1) In the resolution space $\tilde{X}$, the strict transforms $\tilde{C}_i$ for $1 \le i \le N_e$ all share the same homology class $[\tilde{C}] \in H_2(\tilde{X}, \Z)$ because they are strict transforms of curves of the same degree passing through the exact same base locus $\xi$. Therefore, they span a $1$-dimensional subspace in homology. By standard conifold transition theory \cite{C83b,F86} or \cite[Prop. 4.6]{C25} and references there-in, smoothing all $N_e$ nodes increases the third Betti number (and hence $h^{1,2}$) by the number of nodes minus the dimension of the subspace spanned by the exceptional curves, yielding $h^{1,2}(X_t) = h^{1,2}(\tilde{X}) + N_e - 1 = h^{1,2}(X) + N_e - 1$.

For the Picard group, a class $D = aH - \sum_{j=1}^e b_j E_j \in \Pic(\tilde{X})$ descends to a Cartier divisor on $X_0$ if and only if its restriction to the contracted curves $\tilde{C}_i \cong \PP^1$ has degree zero. Since the original curves $C_i$ have degree $e$, $H \cdot \tilde{C}_i = e$. Because they pass transversely through the $e$ blown-up points $\xi$, $E_j \cdot \tilde{C}_i = 1$ for all $j$. Thus, $D \cdot \tilde{C}_i = a(H \cdot \tilde{C}_i) - \sum_{j=1}^e b_j (E_j \cdot \tilde{C}_i) = ae - \sum_{j=1}^e b_j = 0$, giving the stated orthogonal complement.

Finally, to see that this Picard group persists on the general smoothed fiber $X_t$, note that since $X$ is Fano and ordinary double points are rational singularities, we have $H^i(X_0, \OO_{X_0}) = 0$ for $i=1,2$. By the semicontinuity theorem, $H^i(X_t, \OO_{X_t}) = 0$ for a general small smoothing, so the exponential exact sequence yields canonical isomorphisms $\Pic(X_0) \cong H^2(X_0, \Z)$ and $\Pic(X_t) \cong H^2(X_t, \Z)$. Topologically, the local smoothing of a conifold node replaces an $S^2$ with an $S^3$. This creates a new 3-cycle but leaves the second cohomology invariant provided no further 2-cycles are killed. Since the 2-cycles associated to the contracted curves were already eliminated when passing from $\tilde{X}$ to $X_0$, we obtain $H^2(X_t, \Z) \cong H^2(X_0, \Z)$, which gives the desired identification $\Pic(X_t) \cong \Pic(X_0)$.

(2) We will show that $h^0(\tilde{X}, \OO_{\tilde{X}}(\tilde{Y})) = 1$. Consider the standard short exact sequence of sheaves on $\tilde{X}$:
\[
0 \to \OO_{\tilde{X}} \to \OO_{\tilde{X}}(\tilde{Y}) \to \OO_{\tilde{Y}}(\tilde{Y}) \to 0.
\]
Taking the long exact sequence in cohomology, we obtain:
\[
0 \to H^0(\tilde{X}, \OO_{\tilde{X}}) \to H^0(\tilde{X}, \OO_{\tilde{X}}(\tilde{Y})) \to H^0(\tilde{Y}, \OO_{\tilde{Y}}(\tilde{Y})) \to H^1(\tilde{X}, \OO_{\tilde{X}}) \to \dots
\]
Because $X$ is a Fano threefold, $H^1(X, \OO_X) = 0$. Since $\tilde{X}$ is a blow-up of $X$ at points, this vanishing is preserved: $H^1(\tilde{X}, \OO_{\tilde{X}}) = 0$. Furthermore, $H^0(\tilde{X}, \OO_{\tilde{X}}) \cong \C$. Thus, the sequence reduces to:
\[
h^0(\tilde{X}, \OO_{\tilde{X}}(\tilde{Y})) = 1 + h^0(\tilde{Y}, \OO_{\tilde{Y}}(\tilde{Y})).
\]
Therefore, $\tilde{Y}$ is rigid if and only if $h^0(\tilde{Y}, \OO_{\tilde{Y}}(\tilde{Y})) = 0$. 

To evaluate this, we compute the normal bundle $\OO_{\tilde{Y}}(\tilde{Y})$. Recall that by Lemma \ref{lem:BlowHalf}, we have $K_{\tilde{X}} = -2\tilde{Y}$. By the adjunction formula on $\tilde{X}$, the canonical class of $\tilde{Y}$ is:
\[
K_{\tilde{Y}} = (K_{\tilde{X}} + \tilde{Y})\big|_{\tilde{Y}} = (-\tilde{Y})\big|_{\tilde{Y}}.
\]
Hence, $\OO_{\tilde{Y}}(\tilde{Y}) \cong \OO_{\tilde{Y}}(-K_{\tilde{Y}})$. We now compute $h^0(\tilde{Y}, -K_{\tilde{Y}})$. The surface $\tilde{Y}$ is the blow-up of the del Pezzo surface $Y$ at the $e$ generic points $\xi$. Thus, sections of $-K_{\tilde{Y}}$ correspond exactly to sections of $-K_Y$ that vanish at $\xi$:
\[
H^0(\tilde{Y}, -K_{\tilde{Y}}) \cong H^0(Y, -K_Y \otimes \mathcal{I}_{\xi}).
\]
Because $Y$ is a del Pezzo surface of degree $d$, the Riemann-Roch theorem and Kodaira vanishing dictate that $h^0(Y, -K_Y) = K_Y^2 + 1 = d + 1$. Imposing that the sections pass through $e$ generic points imposes $\min(e, d+1)$ independent linear conditions. Consequently:
\[
h^0(\tilde{Y}, -K_{\tilde{Y}}) = \max(0, d + 1 - e).
\]
By our assumption, $e > d$, meaning $e \ge d+1$. Therefore, $h^0(\tilde{Y}, -K_{\tilde{Y}}) = 0$. Returning to our exact sequence, we conclude $h^0(\tilde{X}, \OO_{\tilde{X}}(\tilde{Y})) = 1$. The linear system $\abs{\tilde{Y}}$ contains only the divisor $\tilde{Y}$ itself, proving it is rigid.

\vspace{0.5em}
Now, we study $H - eE_j$. We will prove that $h^0(\tilde{X}, \OO_{\tilde{X}}(H - eE_j)) = 0$. The geometric behavior of the fundamental linear system $\abs{H}$ depends on the degree $d$, so the proof naturally splits into two cases.

\vspace{0.5em}
\noindent\textbf{Case 1: $d \ge 3$.} \\
For $d \in \{3, 4, 5, 8\}$, the linear system $\abs{H}$ is very ample and embeds $X$ into $\PP^{d+1}$. Because it is a closed embedding, the hyperplanes in $\PP^{d+1}$ passing through $p_j$ separate all tangent vectors at $p_j$, and have no other common base points. Consequently, the strict transform linear system $\abs{H - E_j}$ restricts to the exceptional divisor $E_j \cong \PP^2$ as the complete linear system of lines $\abs{\OO_{\PP^2}(1)}$, and is globally generated on the entire blow-up $\tilde{X}$.

By Bertini's theorem, the intersection of two general members $S_1, S_2 \in \abs{H - E_j}$ is a smooth, irreducible mobile curve $C$. Suppose, for contradiction, that there exists an effective divisor $D \in \abs{H - eE_j}$. Because $C$ is a general mobile curve, it is not contained in $D$, so their intersection number must be non-negative: $D \cdot C \ge 0$.

The homological class of the curve is $[C] = (H - E_j)^2 = H^2 - 2HE_j + E_j^2$. Because $H$ is a pullback from $X$, $H \cdot E_j = 0$, reducing the class to $[C] = H^2 + E_j^2$. Intersecting this with $D = H - eE_j$ yields:
\[
D \cdot C = (H - eE_j) \cdot (H^2 + E_j^2) = H^3 + HE_j^2 - eE_jH^2 - eE_j^3.
\]
Substituting the known intersection numbers $H^3 = d$ and $E_j^3 = 1$, we obtain:
\[
D \cdot C = d - e.
\]
Since our hypothesis states $e > d$, this implies $D \cdot C < 0$, which contradicts the geometric requirement that $D \cdot C \ge 0$. Thus, the linear system $\abs{H - eE_j}$ must be empty.

\vspace{0.5em}
\noindent\textbf{Case 2: $d = 2$.} \\
For the quartic double solid ($d=2$), the linear system $\abs{H}$ is not an embedding. Instead, it defines a $2:1$ finite morphism $\phi: X \to \PP^3$ branched along a smooth quartic surface. 

Instead, we use a local differential argument. The global sections of $H$ are precisely the pullbacks of linear forms from $\PP^3$. Because $p_j$ is chosen as a general point, it avoids the ramification locus of $\phi$. Thus, $\phi$ is locally an isomorphism (étale) in a neighborhood of $p_j$. 

A non-zero linear form on $\PP^3$ vanishing at the point $\phi(p_j)$ vanishes to order exactly $1$. Because $\phi$ is a local isomorphism at $p_j$, its pullback to $X$ must also vanish to order exactly $1$ at $p_j$. Therefore, no non-zero section of $H$ can vanish to order $2$ or higher at an unramified point. 

This implies that $H^0(\tilde{X}, \OO_{\tilde{X}}(H - eE_j)) = 0$ for all $e \ge 2$. Because our assumption $e > d$ requires $e \ge 3$, the linear system is empty.
\end{proof}

Now, we modify the standard conifold transition to reduce the Picard rank of the singular space. By simultaneously contracting disjoint rational curves of smaller degrees, we can force the Picard group of the singular threefold to be generated only by the boundary divisor $Y_0$. 

\begin{prop}\label{prop:MultDegrees}
Let $X$ be a smooth prime Fano threefold of index two and degree $d$. Let $\xi = \{\xi_1, \dots, \xi_e\} \subset Y$ be a set of generic points on a generic smooth divisor $Y \in \abs{-\frac{1}{2}K_X}$. Let $\sigma \colon \tilde{X} \to X$ be the blow-up at $\xi$ with exceptional divisors $E_1, \dots, E_e$. Assume there exist:
\begin{enumerate}
    \item A smooth rational curve $\tilde{C} \subset \tilde{X}$, which is the strict transform of a curve of degree $e$ in $X$ passing transversely through all points of $\xi$.
    \item For each $k \in \{1, \dots, e\}$, a smooth rational curve $\tilde{C}_k \subset \tilde{X}$, which is the strict transform of a curve of degree $e-1$ passing transversely through the subset $\xi \setminus \{\xi_k\}$.
\end{enumerate}
Assume further that $\tilde{C}$ and all $\tilde{C}_k$ are mutually disjoint and have normal bundle $\OO(-1) \oplus \OO(-1)$. Let $\pi \colon \tilde{X} \to X_0$ be the contraction of all $e+1$ curves. Then:
\begin{itemize}
    \item[i)] The Picard group of $X_0$ has rank $1$ and is generated by $Y_0 = \pi_*\tilde{Y}$.
    \item[ii)] If $e > d + 1$, then $H^1(X_0, \OO_{X_0}(Y_0)) \cong H^1(Y_0, N_{Y_0/X_0}) \neq 0$.
    \item[iii)] For a generic global smoothing $X_t$ of $X_0$ (not preserving the pair structure), the line bundle $\OO_{X_0}(Y_0)$ deforms to a line bundle $\mathcal{L}_t \in \Pic(X_t)$, but $H^0(X_t, \mathcal{L}_t) = 0$. Consequently, $X_t$ contains no effective surfaces.
\end{itemize}
\end{prop}

\begin{proof}
The Picard group of $\tilde{X}$ is generated by $H = \sigma^*\OO_X(1)$ and the exceptional divisors $E_1, \dots, E_e$. Let $D = aH - \sum_{i=1}^e b_i E_i \in \Pic(\tilde{X})$. For $D$ to descend to a Cartier divisor on $X_0$, it must intersect all contracted curves at degree zero. 

Because the original curve for $\tilde{C}$ has degree $e$ and passes through all $p_i$, we have $H \cdot \tilde{C} = e$ and $E_i \cdot \tilde{C} = 1$ for all $i$. Thus:
\[
D \cdot \tilde{C} = ae - \sum_{i=1}^e b_i = 0.
\]
Because the original curve for $\tilde{C}_k$ has degree $e-1$ and passes through all $p_i$ except $p_k$, we have $H \cdot \tilde{C}_k = e-1$, $E_k \cdot \tilde{C}_k = 0$, and $E_i \cdot \tilde{C}_k = 1$ for $i \neq k$. Thus:
\[
D \cdot \tilde{C}_k = a(e-1) - \sum_{i \neq k} b_i = 0.
\]
Let $S = \sum_{i=1}^e b_i$. The first equation gives $S = ae$. The second equation gives $\sum_{i \neq k} b_i = S - b_k = a(e-1)$. Substituting $S = ae$ yields:
\[
ae - b_k = ae - a \implies b_k = a.
\]
Since this holds for every $k \in \{1, \dots, e\}$, any Cartier divisor on $X_0$ must pull back to a divisor of the form $D = aH - a \sum_{i=1}^e E_i = a\tilde{Y}$. Therefore, $\Pic(X_0)$ is generated by $Y_0 = \pi_*\tilde{Y}$, and its rank is $1$.

\vspace{0.5em}
We evaluate the normal bundle $N_{Y_0/X_0}$. Because $\tilde{Y}$ avoids the exceptional locus of $\pi$, $Y_0$ lies in the smooth locus of $X_0$ and $N_{Y_0/X_0} \cong N_{\tilde{Y}/\tilde{X}}$. By adjunction on $\tilde{X}$, since $K_{\tilde{X}} + 2\tilde{Y} \sim 0$, we have $N_{\tilde{Y}/\tilde{X}} \cong \OO_{\tilde{Y}}(\tilde{Y}) \cong \OO_{\tilde{Y}}(-K_{\tilde{Y}})$.

The surface $\tilde{Y}$ is the blow-up of a degree $d$ del Pezzo surface at $e$ points. For $e > d$, $-K_{\tilde{Y}}$ is not effective, so $H^0(Y_0, N_{Y_0/X_0}) = H^0(\tilde{Y}, -K_{\tilde{Y}}) = 0$. This forces $h^0(X_0, \OO_{X_0}(Y_0)) = 1$ (only the trivial section exists). 
By the Riemann-Roch theorem on the surface $\tilde{Y}$:
\[
\chi(\tilde{Y}, -K_{\tilde{Y}}) = \frac{1}{2}(-K_{\tilde{Y}} \cdot (-2K_{\tilde{Y}})) + \chi(\OO_{\tilde{Y}}) = K_{\tilde{Y}}^2 + 1 = d - e + 1.
\]
Since $h^0(-K_{\tilde{Y}}) = 0$ and $h^2(-K_{\tilde{Y}}) = h^0(2K_{\tilde{Y}}) = 0$, we conclude that $h^1(Y_0, N_{Y_0/X_0}) = e - d - 1$. If $e > d + 1$, then $H^1(Y_0, N_{Y_0/X_0}) \neq 0$. 
From the standard short exact sequence $0 \to \OO_{X_0} \to \OO_{X_0}(Y_0) \to N_{Y_0/X_0} \to 0$ and the vanishing of $H^1(\OO_{X_0})$, we obtain an isomorphism $H^1(X_0, \OO_{X_0}(Y_0)) \cong H^1(Y_0, N_{Y_0/X_0}) \neq 0$. This non-trivial cohomology group governs the obstructions to extending the section of $Y_0$ to first-order deformations of $X_0$, see Subsection \ref{ss:setting}.

\vspace{0.5em}
Because $H^i(\OO_{X_0}) = 0$ for $i > 0$, the line bundle $\OO_{X_0}(Y_0)$ lifts unobstructed to a line bundle $\mathcal{L}_t \in \Pic(X_t)$ on any small global smoothing. Since $\Pic(X_0)$ has rank $1$, $\Pic(X_t)$ is generated by $\mathcal{L}_t$.
 Since $H^1(X_0, \OO_{X_0}(Y_0)) \neq 0$, the unique global section of $Y_0$ is obstructed from deforming along a generic moduli direction of $X_0$ (i.e., one that does not preserve the log pair structure). By semicontinuity, for a general $t$, the dimension of the space of global sections drops from $1$ to $0$, meaning $H^0(X_t, \mathcal{L}_t) = 0$. Because $\mathcal{L}_t$ generates the Picard group and has no sections, the threefold $X_t$ possesses no effective surfaces.
\end{proof}

\begin{ex}
  We can take $X$ to be a cubic threefold and take $e=1,2$, this is, we consider lines and conics passing through generic points $\xi_1,\xi_2$. Using residuation, we can show that the hypothesis of Proposition \ref{prop:MultDegrees} are satisfied.
\end{ex}
Finally, we want to show that general very free curves avoid the conifold locus and persist.
\begin{prop}
\label{prop:avoidance_persistence}
Let $X$ be a smooth prime Fano threefold of index $2$, let $Y \in |-\frac{1}{2}K_X|$ be a generic smooth divisor, let
$\sigma:\tilde X=Bl_{\xi}X\to X$ be the blow-up at a very general configuration
$\xi\subset Y$, and let
\[
\pi:\tilde X\to X_0
\]
be the contraction of the finite union
\[
\Gamma=C_1\cup\cdots\cup C_N\subset \tilde X
\]
of strict disjoint transforms of the anchored curves through $\xi$.

Let $\mathcal H_0$ be an irreducible component of $\overline M_{0,0}(X,e')$
whose general point parametrizes a smooth embedded very free rational curve
$C\subset X$. Then, for a general curve $C\in \mathcal H_0$, the following hold:
\begin{enumerate}
\item $C$ is disjoint from $\xi$ and from the finite union
\[
B:=\sigma(C_1)\cup\cdots\cup\sigma(C_N)\subset X;
\]
\item the strict transform $\tilde C\subset \tilde X$ is disjoint from $\Gamma$;
\item if $C_0:=\pi(\tilde C)\subset X_0$ is the corresponding curve on the singular
fiber, then $C_0\simeq C$ and
\[
N_{C_0/X_0}\cong N_{C/X};
\]
in particular,
\[
H^1(C_0,N_{C_0/X_0})=0;
\]
\item for any flat smoothing $f:\mathcal X \to \Delta$ with central fiber $X_0=f^{-1}(0)$, the curve $C_0$ deforms to nearby smooth fibers.
\end{enumerate}
\end{prop}

\begin{proof}
Let $\mathcal H_0^\circ\subset \mathcal H_0$ be a dense open subset over which the
universal curve
\[
p:\mathcal U^\circ\to \mathcal H_0^\circ
\]
is smooth and the evaluation morphism
\[
\ev:\mathcal U^\circ\to X
\]
is smooth. Such an open subset exists because the general curve in
$\mathcal H_0$ is very free.

We first prove that a general curve in $\mathcal H_0^\circ$ avoids $\xi$.
Since $\xi$ is finite, each point $x\in \xi$ has codimension $3$ in the
threefold $X$. Because $\ev$ is smooth on $\mathcal U^\circ$, the inverse image
$\ev^{-1}(x)$ has codimension $3$ in $\mathcal U^\circ$. Hence
\[
\dim \ev^{-1}(x)\le \dim \mathcal U^\circ-3
= (\dim \mathcal H_0^\circ+1)-3
= \dim \mathcal H_0^\circ-2.
\]
Therefore $\ev^{-1}(x)\to \mathcal H_0^\circ$ is not dominant. Since $\xi$ is
finite, the union $\bigcup_{x\in \xi}\ev^{-1}(x)$ is a proper closed subset of
$\mathcal H_0^\circ$. Thus a general curve in $\mathcal H_0$ does not meet
$\xi$.

Now let
\[
B=\sigma(C_1)\cup\cdots\cup\sigma(C_N)\subset X.
\]
Each $B_i=\sigma(C_i)$ is a smooth rational curve in $X$, hence a
codimension-$2$ closed subset of the threefold $X$. Since $\ev$ is smooth on
$\mathcal U^\circ$, each inverse image $\ev^{-1}(B_i)$ has codimension $2$ in
$\mathcal U^\circ$. Thus
\[
\dim \ev^{-1}(B_i)\le \dim \mathcal U^\circ-2
= (\dim \mathcal H_0^\circ+1)-2
= \dim \mathcal H_0^\circ-1,
\]
so $\ev^{-1}(B_i)\to \mathcal H_0^\circ$ is not dominant. Because there are
only finitely many $B_i$, the union of the corresponding loci is still a proper
closed subset of $\mathcal H_0^\circ$. Hence a general curve in $\mathcal H_0$
is disjoint from $B$.

For such a general curve $C$, the blow-up $\sigma:\tilde X\to X$ is an
isomorphism along $C$ because $C\cap \xi=\emptyset$. Therefore its strict
transform $\tilde C$ is canonically isomorphic to $C$. Moreover, since
$C\cap B=\emptyset$, the curve $\tilde C$ is disjoint from $\Gamma$, because
$C_i$ maps isomorphically to $B_i$ away from the blown-up points. Hence
$\pi$ is an isomorphism in a neighborhood of $\tilde C$, and the image
\[
C_0:=\pi(\tilde C)\subset X_0
\]
is a smooth rational curve with
\[
N_{C_0/X_0}\cong N_{\tilde C/\tilde X}\cong N_{C/X}.
\]

Since the general curve in $\mathcal H_0$ is very free, its normal bundle is
ample, so in particular
\[
H^1(C,N_{C/X})=0.
\]
Therefore
\[
H^1(C_0,N_{C_0/X_0})=0.
\]
Now apply the standard normal bundle sequence for the inclusion
$C_0\subset X_0\subset \mathcal X$ in any flat smoothing $f: \mathcal X\to\Delta$. Because $X_0 = f^{-1}(0)$ is a principal divisor (the central fiber), its normal bundle $N_{X_0/\mathcal X}$ is trivial. Restricting this to $C_0$ yields the sequence:
\[
0\to N_{C_0/X_0}\to N_{C_0/\mathcal X}\to \OO_{C_0}\to 0.
\]
Since $C_0 \simeq \mathbb{P}^1$, $H^1(C_0, \OO_{C_0}) = 0$. Thus, the vanishing of $H^1(C_0,N_{C_0/X_0})$ implies $H^1(C_0,N_{C_0/\mathcal X})=0$,
so the relative Hilbert scheme is smooth at $[C_0]$. The induced map to the base
has surjective differential because
\[
H^0(C_0,N_{C_0/\mathcal X})\twoheadrightarrow H^0(C_0,\OO_{C_0})
\cong T_0\Delta.
\]
Hence the component through $[C_0]$ dominates $\Delta$, and $C_0$ deforms to
curves in nearby smooth fibers.
\end{proof}

\section{Hodge theoretic properties}
\label{s:hodge}

In this section, we study the Hodge theory of the smoothed general fiber $X_t$. The main tool is the limiting mixed Hodge structure (LMHS) on the SNC model of the central fiber $X_0$ and the recent results of Chen \cite{Che24} on global smoothings of normal crossing varieties.

Recall briefly Chen's theorem. Let $f \colon \mathcal{X} \to \Delta$ be a one-parameter smoothing of a threefold $X_0$ with ordinary double points. To compute the Hodge structures, one passes to the simple normal crossing (SNC) resolution $Y = \tilde{X}_0 \cup E_1 \cup \dots \cup E_k/ \sim$, where $\tilde{X}_0$ is the blow-up of $X_0$ at the nodes (equivalently at the curves $C_i\subset \tilde X$), the $E_i$ are smooth 3-dimensional quadrics and $\sim$ denotes the equivalence relation obtained by gluing $E_i$ to $\tilde X_0$ along the exceptional quadrics. Chen proves in \cite[Theorem 4.1]{Che24} that if the components of $Y$ are K\"ahler, then the general fiber $X_t$ satisfies the $\partial\bar{\partial}$-lemma. 

Because our construction starts with a Fano threefold $X$, the blow-up $\tilde{X}_0$ is projective and hence K\"ahler. Furthermore, by Corollary \ref{cor:independent_smoothings}, the nodes can be smoothed independently without global homological obstructions. As an immediate corollary, we deduce:

\begin{cor}
The general smoothed fiber $X_t$ satisfies the $\partial\bar{\partial}$-lemma and inherits a polarized pure Hodge structure of weight 3.
\end{cor}

\begin{ex}
Let $X$ be a smooth cubic threefold. Its intermediate Jacobian $J(X)$ is a 5-dimensional principally polarized abelian variety. Suppose we smooth $k$ nodes with $r$ relations to obtain $X_t$. The third Betti number is $b_3(X_t) = 10 + 2(k-r)$.

We can ask if $X_t$ could simply be isomorphic to a space $\tilde{Z}_t$, obtained by blowing up a curve $C_t$ in a twistor space $Z \to M$. As $X$ is simply connected, we can assume the base $M$ is any simply connected Riemannian 4-manifold (for instance, $S^4$ or $n\C P^2$). Because $M$ is simply connected, $H^1(M, \Z) = 0$. By Poincar\'e duality on the 4-manifold, $H^3(M, \Z) = 0$ as well. The Leray spectral sequence for the $S^2$-fibration then forces $H^3(Z, \C) = 0$.

If we blow up a smooth curve $C_t$ of genus $g = 5 + (k-r)$ inside $Z$, the blow-up formula gives $b_3(\tilde{Z}_t) = 2g = 10 + 2(k-r)$. Thus, $X_t$ and $\tilde{Z}_t$ are simply connected and have the exact same third Betti number. Topologically, $H^3$ cannot tell them apart.

However, they fail to match algebraically when we take the limit as $t \to 0$. If $X_t \cong \tilde{Z}_t$, the intermediate Jacobian of the blown-up twistor space is the Jacobian of the curve: $J(X_t) \cong \text{Jac}(C_t)$. 

On the threefold side, by the LMHS description (see for example \cite{F19}), the compact abelian part of the limiting intermediate Jacobian $J_{\text{lim}}(X_t)$ is constructed from $Gr^W_3 H^3_{\text{lim}}$. This is exactly $J(X)$, the intermediate Jacobian of the smooth cubic threefold.

On the twistor side, consider the family of curves $C_t$. By the semi-stable reduction theorem, after possibly a finite base change $t \mapsto t^d$, the family $C_t$ can be completed over $t=0$ to a nodal stable curve $C_0$. This base change does not affect the weight filtration of the LMHS, so the compact abelian part of the limit remains unchanged. The limit of the Jacobians $\text{Jac}(C_t)$ is the generalized Jacobian of the stable curve $C_0$. The compact abelian part of this generalized Jacobian is a product of Jacobians of smooth curves, $\prod_i \text{Jac}(\tilde{C}_{0, i})$, where the $\tilde{C}_{0, i}$ are the smooth normalizations of the irreducible components of $C_0$. 

Comparing the compact abelian parts of both limits, we would be forced to have $J(X) \cong \prod_i \text{Jac}(\tilde{C}_{0, i})$. But by the Clemens-Griffiths theorem \cite{CG72}, the cubic threefold's intermediate Jacobian $J(X)$ is indecomposable and can not be obtained as a Jacobian of a curve (or a product of them). This is a contradiction. This proves that our smoothings cannot be obtained as blow-ups on curves of any simply connected twistor space.
\end{ex}

\bibliographystyle{alpha}
\bibliography{sample}

\end{document}